\documentclass[11pt]{article}

\usepackage{todonotes}
\usepackage{amssymb}
\usepackage{amsthm}
\usepackage{amsmath}
\usepackage{graphicx}
\usepackage{mathtools}
\usepackage{thmtools,hyperref}

\usepackage{subfig}
\usepackage{cases}
\usepackage{mathrsfs}
\usepackage{soul}

\newtheorem{theorem}{Theorem}
\newtheorem{cor}{Corollary}[section]
\newtheorem{lemma}[cor]{Lemma}
\newtheorem{prop}[cor]{Proposition}
\newtheorem{defn}[cor]{Definition}
\newtheorem{rem}[cor]{Remark}

\newcommand{\tr}{\operatorname{tr}}

\newcommand{\perIntv}{[-\pi,\pi]}
\newcommand{\perIntegral}{\int_{-\pi}^{\pi}}

\newcommand{\yobs}{y^{\mathrm{obs}}}

\newcommand{\hid}[1]{}
\newcommand{\Xspace}{\mathbb{X}}
\newcommand{\Yspace}{\mathbb{Y}}
\newcommand{\calO}{\mathcal{O}}
\newcommand{\paren}[1]{\left(#1\right)}
\newcommand{\bracket}[1]{\left[#1\right]}

\usepackage{dsfont}
\newcommand{\N}{\mathds{N}}
\newcommand{\Z}{\mathds{Z}}
\newcommand{\R}{\mathds{R}}

\let\RE\Re
\let\Re=\undefined
\DeclareMathOperator{\Re}{\RE e}
\let\IM\Im
\let\Im=\undefined
\DeclareMathOperator{\Im}{\IM m}

\DeclareMathOperator{\trace}{trace}
\DeclareMathOperator{\argmin}{argmin}

\newcommand{\iunit}{\mathrm i}

\newcommand{\norm}[1]{\left\| #1\right\|}

\newcommand{\F}{\mathcal{F}}

\newcommand{\ourtitle}{Density Matrix Reconstructions in Ultrafast Transmission Electron Microscopy:
Uniqueness, Stability, and Convergence Rates}
\title{\ourtitle}

\author{
Cong Shi\thanks{\footnotesize Institute of Numerical and Applied Mathematics, University of G\"ottingen}
 \and Claus Ropers \thanks{\footnotesize IV. Physical Institute, University of G\"ottingen}
\and Thorsten Hohage\thanks{\footnotesize Institute of Numerical and Applied Mathematics, University of G\"ottingen}}

\usepackage{dsfont}
\begin{document}
\date{}
\maketitle

\begin{abstract}
In the recent paper \cite{PriRatYalHohFeiSchRop17} the first experimental
determination of the density matrix of a free electron beam has been reported.
The employed method leads to a
linear inverse problem with a positive semidefinite operator as unknown.
The purpose of this paper is to complement the experimental and algorithmic results
in the work mentioned above by a mathematical analysis of the inverse problem concerning uniqueness,
stability, and rates of convergence under different types of \textit{a-priori} information.
\end{abstract}

\medskip

\bigskip


\noindent {\footnotesize Keywords: Tikhonov regularization, variational source
conditions, uniqueness, stability, electron microscopy, SQUIRRELS}

\section{Introduction}
The density matrix is a fundamental notion in quantum statistics which describes the
statistical state of an ensemble of identical 
single or many body quantum systems. It is a positive semidefinite
operator of trace $1$ on the Hilbert space describing the state of a single quantum system.
In the area of quantum optics, there are well-established techniques for characterizing the quantum state of the electromagnetic field in terms of its density matrix
\cite{PR:04,schleich:15}. Such `quantum state tomography' facilitates the discrimination of, for example, coherent states, squeezed states, thermal states or photon number (Fock) states.
In contrast, the reconstruction of the quantum state of a beam of free electrons has only recently been established using inelastic electron-light scattering \cite{PriRatYalHohFeiSchRop17}.
The reconstruction technique, termed `SQUIRRELS' for  `Spectral Quantum Interference for the Regularized Reconstruction of free ELectron States'), is experimentally based on the principle of `Photon-Induced Near Field Electron Microscopy' (PINEM) \cite{BFZ:09}.
In PINEM, a beam of electrons is passed through the near field of laser-illuminated nanostructures or thin films, leading to the formation of sidebands in the electron energy spectrum, spaced by the photon energy
\cite{PLZ:10,abajo:10}. The spatially varying number of created sidebands yields the optical field strength with very high resolution on the nanometer scale
\cite{BFZ:09,piazza_etal:15}. However, the quantum coherent nature of the electron-light interaction has
also led to the observation of other fundamental quantum effects, such as multilevel Rabi oscillations
\cite{feist_etal:15,AAK:10} or Ramsey-type phase interference in spatially separated fields \cite{EFSR:16}. Moreover, it has recently been shown that the interaction can be used to temporally structure electron beams into a train of attosecond pulses, the duration of which was determined by SQUIRRELS \cite{PriRatYalHohFeiSchRop17}. The various existing and future applications of inelastic electron-light scattering and the relevance of the specific electron state resulting from such interactions calls for a solid mathematical basis underlying the quantum state reconstruction scheme.
The mathematical aspects of SQUIRRELS involve a linear inverse problem with the density matrix as unknown. In \cite{PriRatYalHohFeiSchRop17} this inverse problem was solved by Tikhonov regularization with positive semidefiniteness and trace constraints using quadratic semi-definite programming. The purpose of the present paper is to provide mathematical foundations of the SQUIRRELS method.

In the experiment in \cite{PriRatYalHohFeiSchRop17}, light reflection from a thin graphite film mediates the
 interaction of free electrons with laser photons of two frequencies $\omega$ and $2\omega$ and a controllable
 relative phase $\theta$.
%
%
As the interactions of electrons with
laser photons lead to a comb-type energy spectrum of the electrons separated by the photon
energy, the Hilbert space describing the state of an electron may be chosen as $l^2(\Z)$.
The effect of the interaction of an $\omega$-photon with a single electron is described
by a unitary operator $U_{\omega}(\theta):l^2(\Z)\to l^2(\Z)$ given in matrix representation by
\begin{equation}\label{U-omega}
(U_{\omega}(\theta))_{k, l} = e^{\iunit(k-l)\theta}J_{k-l}(2|g_\omega|),\qquad k,l\in\Z.
\end{equation}
Here $J_{l}(2|g_\omega|)$ denotes the Bessel function of the first kind of order $l$,
and $g_\omega$ is a coupling constant associated with the laser.
The effect of the photon-electron interaction on the free-electron density matrix
$\rho = (\rho_{j,k})_{j,k\in\Z}$ is then described by
\[
\rho_{\text{out}}(\theta) = U_{\omega}(\theta)\rho U^*_{\omega}(\theta),
\]
where $\rho$ and $\rho_{\text{out}}(\theta)$ are the density matrices before and after the interaction,
respectively. However, only the diagonal values $\rho_{\text{out}}(\theta)$ are
observable. On the other hand, since the phase parameter $\theta\in\perIntv$ is experimentally
controllable, we may observe a spectrogram
$y(\theta,l)=\paren{\rho_{\text{out}}(\theta)_{ll}}_{l\in\Z}$ for each value of $\theta$.
The inverse problem to find the electron density matrix $\rho$ from the measured data $y$
is then described by the operator equation
\begin{equation}\label{eq:opeq}
T\rho = y
\end{equation}
with a bounded linear forward operator $T:\Xspace\to \Yspace$
between Hilbert spaces $\Xspace:=l^2(\Z\times\Z)$
and $\Yspace:= L^2(\Z\times\perIntv)$ given by
\begin{equation}\label{OperatorTDef}
(T\rho)(l, \theta) = p(l, \theta) = ( U_{\omega}(\theta)\rho U^*_{\omega}(\theta))_{l, l}.
\end{equation}
The aim of this paper is to analyze this inverse problem mathematically concerning
uniqueness, stability, and rates of convergence.

The plan of the remainder of this paper is organized as follows: In section \ref{sec:main_results}
we formulate our main results. Section \ref{sec:uniqueness} is devoted to the
proof that $T$ is injective. It is based on a factorization of $T$ which is also
fundamental for the rest of this paper. Our main tool for the proofs of stability estimates
and convergence rates  are variational source conditions,
which will be treated in section \ref{sec:vsc}.

\section{Main results}\label{sec:main_results}
Our first main result asserts that the unknown density matrix is in fact uniquely determined by
the experimental data in the absence of noise and modelling errors:

\begin{theorem}[uniqueness]\label{thUniqueness}
The operator $T$ defined in \eqref{OperatorTDef} is injective.
\end{theorem}

It follows from our analysis (see Corollary \ref{coro:illposedness}) and has been observed in numerical
experiments that the inverse problem \eqref{eq:opeq} is ill-posed. Therefore, a natural question concerns
the degree of ill-posedness or the degree of stability that can be obtained under certain types of a-priori
information on the true density matrix $\rho^{\dagger}$.
This will be addressed in the following three theorems.

Moreover, to obtain stable reconstruction
for noisy experimental data $\yobs$, some kind of regularization has to be employed.
In \cite{PriRatYalHohFeiSchRop17} constrained Tikhonov regularization of the following form has been used:
\begin{equation}\label{eq:sol_Tikhonov}
\widehat{\rho}_{\alpha}=\arg\min_\rho{\|T\rho-\yobs\|_\Yspace^2+\alpha\|\rho\|_\Xspace^2}
\qquad \mbox{subject to } \rho\geq 0,\;\trace(\rho)=1
\end{equation}
Here we minimize only over density matrices, i.e.\
positive semidefinte operators of trace $1$.
The regularization parameter $\alpha>0$ is chosen
by the discrepancy principle as follows: Let $\delta>0$ be the deterministic noise level, i.e.
\[
\norm{T\rho^\dagger-\yobs}_\Yspace<\delta
\]
for the true density matrix $\rho^\dagger$. Then $\alpha>0$ is chosen such that
\begin{equation}\label{eq:disc}
\delta\leq\norm{T\widehat{\rho}_{\alpha}-\yobs}_{\Yspace}\leq\tau\delta
\end{equation}
for some $\tau>1$. We will also derive error bounds for Tikhonov regularization described by
\eqref{eq:sol_Tikhonov} and \eqref{eq:disc}.

We first consider band-limited density matrices:
\begin{theorem}[H\"older-type estimates for band limited $\rho$]\label{rates1}
Suppose the density matrices $\rho^{(1)}$, $\rho^{(2)}$, and
$\rho^\dagger$ are band-limited, i.e.\
\begin{equation*}
\rho_{n+k, n}=0, \qquad \mbox{for all } |k|>k_0
\end{equation*}
for $\rho\in\{\rho^{\dagger},\rho^{(1)},\rho^{(2)}\}$
and some $k_0>0$. Then the stability estimate
\begin{equation}\label{eq:Hoelder_stability}
\norm{\rho^{(1)}-\rho^{(2)}}
\leq C\norm{T\rho^{(1)}-T\rho^{(2)}}^{\frac{1}{1+2k_0}}
\end{equation}
holds true for some constant $C$ depending only on $k_0$ and $|g_{\omega}|$.
Moreover, for the Tikhonov regularized solution $\widehat{\rho}_{\alpha}$ given by
\eqref{eq:sol_Tikhonov} with parameter choice rule $\alpha \sim\delta^\frac{2+2k_0}{1+2k_0}$,
or with $\alpha$ chosen by the discrepancy principle \eqref{eq:disc}, the error bound
\begin{equation}\label{eq:Hoelder_rate}
\norm{\widehat{\rho}_{\alpha}-\rho^\dagger} \leq C\delta^{\frac{1}{1+2k_0}}
\end{equation}
is satisfied for all $\delta\in (0,1]$
with $C$ independent of $\rho^\dagger$.
\end{theorem}

If we relax band-limitation to a polynomial or exponential
decay condition, we
only obtain slower than H\"older stability estimates and convergence rates:

\begin{theorem}[Sub-H\"older rates under decay conditions]\label{rates2}
Suppose there exists $C_\rho>0$ such that the off-diagonal entries of the density matrices  $\rho^{(1)}$, $\rho^{(2)}$, and
$\rho^\dagger$ satisfy either a exponential or a
polynomial decay condition:
\begin{subequations}
\begin{align}\label{Log-Assum}
&\sum_{n=-\infty}^{\infty}|\rho_{n+k,n}| \leq C_\rho|k|^{-\frac{1}{2}-2\mu}
\quad \mbox{for some $\mu>0$ or}\\
\label{Hol-Assum2}
&\sum_{n=-\infty}^{\infty}|\rho_{n+k,n}| \leq C_\rho b^{|k|}
\quad\mbox{ for some }b<1
\end{align}
\end{subequations}
and for all $k\in\mathbb{Z}$
and $\rho\in\{\rho^{(1)}, \rho^{(2)}, \rho^{\dagger}\}$.
Then the stability estimate
\begin{equation*}
\norm{\rho^{(1)}-\rho^{(2)}}\leq
\Phi\paren{\norm{T\rho^{(1)}-T\rho^{(2)}}}
\end{equation*}
holds true for $\norm{T\rho^{(1)}-T\rho^{(2)}}\leq \frac{1}{2}$, and for the Tikhonov regularized solution
$\widehat{\rho}_{\alpha}$ given by \eqref{eq:sol_Tikhonov} with
$\alpha$ chosen by the discrepancy principle \eqref{eq:disc}, the error is bounded by
\[
\norm{\widehat{\rho}_{\alpha}-\rho^\dagger}\leq 2(1+\tau)\Phi(\delta)
\]
where the function $\Phi$ is given by
\begin{align*}
\Phi(\delta):=
\begin{cases}
C\left(\frac{-\log \delta}{\log(-\log \delta)}\right)^{-2\mu}& \mbox{in case of \eqref{Log-Assum}},\\
C\exp\left(-\sqrt{(-\log\delta)(-\log b)}\right)&
 \mbox{in case of \eqref{Hol-Assum2}},
\end{cases}
\end{align*}
for $\delta\in(0,1/2]$
with some constant $C$ depending only on $C_\rho$ and $|g_{\omega}|$, $\mu$, and $b$.
\end{theorem}
Note that the logarithmic rate
$\paren{\frac{-\log\delta}{\log(-\log\delta)}}^{-\mu}$
for the polynomial decay condition \eqref{Log-Assum}
is slower than $\mathcal{O}((-\log\delta)^{-\mu})$
as $\delta\to 0$,
but faster than
$\mathcal{O}((-\log\delta)^{-\mu'})$ for any
$\mu'<\mu$.
On the other hand,
the rate for the exponential decay condition
\eqref{Hol-Assum2} is slower that any H\"older
rate $\mathcal{O}(\delta^{\nu})$ for $\nu>0$, but faster
that any logarithmic rate. Such rates of convergence and
stability estimates occur much less frequently
than H\"older and logarithmic rates, but similar rates
have been derived e.g.\ in scattering theory for the reconstruction of a near field data from far field data
(see \cite[Lemma 4.2]{HH:01}).

\section{Uniqueness}\label{sec:uniqueness}
The aim of this section is to prove Theorem \ref{thUniqueness}.
Throughout this paper we use the following notations:

Let $\F:L^2(\perIntv) \to \ell^2(\Z)$,
\begin{align*}
(\F f)_n&=\frac{1}{\sqrt{2\pi}}\perIntegral e^{\iunit n\theta}f(\theta)\,d\theta, \qquad f\in L^2\perIntv, n\in\Z
\end{align*}
denote the periodic Fourier transform. Since the scaling factor is chosen
such that $\F$ is unitary, we have
\begin{align*}
(\F^{-1}a)(\theta)
&=(\F^{*}a)(\theta)=\frac{1}{\sqrt{2\pi}}\sum_{n\in\Z}e^{\iunit n\theta}a_n, \qquad
a\in\ell^2(\Z), \theta\in \perIntv.
\end{align*}
Recall the periodic Fourier convolution theorem
\begin{equation}\label{eq:Fourier_convolution}
2\pi\sum_{n\in\Z}e^{i\varphi n} a_nb_n = \perIntegral(\F^*a)(\tau) (\F^*b)(\varphi-\tau)\,d\tau,
\qquad \varphi\in\perIntv
\end{equation}
for $a,b\in \ell^2(\Z)$.
Fourier transforms on spaces of multi-variate functions will be labeled by superscript(s)  indicating the position of the variable(s) on which they act. For example,
\begin{align}
&\begin{aligned}\label{FourierTransformF1}
&{\F^{(*1)}}:\ell^2(\Z^2)\to L^2(\perIntv\times\Z)\\
&({\F^{*(1)}} a)(\theta,m):=\frac{1}{\sqrt{2\pi}}\sum_{n\in\Z}e^{\iunit n\theta}a_{n,m},
\end{aligned}\\
&\begin{aligned}\label{FourierTransformF2}
&{\F^{(1,*2)}}:L^2(\perIntv\times\Z)\to L^2(\Z\times\perIntv)\\
&({\F^{(1,*2)}} f)(n,\varphi):=\frac{1}{2\pi}\sum_{n\in\Z}\perIntegral e^{\iunit n\theta}e^{\iunit \varphi m}f(\theta,m)\,d\theta. 
\end{aligned}
\end{align}
It is easy to see that these operators are again unitary. (This can either be proved directly
or by noting that they are tensor products of unitary operator,
$\F^{*1}=\F^*\otimes I$ and $\F^{(1,*2)}= \F\otimes \F^*$.) In particular,
\begin{equation}\label{eq:mixedFTinv}
(\F^{(1,*2)})^{-1} = (\F^{(1,*2)})^{*}=\F^{(*1,2)}.
\end{equation}

Our basic tool for the uniqueness proof and also for the following sections is the following
factorization of the operator $T$:
\begin{prop}\label{prop:Tfactorization}
The operator $T$ defined in \eqref{OperatorTDef} has the factorization
\begin{equation}\label{eq:Tfactorization}
T=\F^{(*1,2)}M_m\F^{(1)}\mathcal{G}.
\end{equation}
Here $\F^{(1)}$ and $\F^{(*1,2)}$ are defined in \eqref{FourierTransformF1}--\eqref{eq:mixedFTinv},
$M_m:L^2(\perIntv\times\Z)\to L^2(\perIntv\times\Z)$ is the multiplication operator
\begin{align}\label{eq:Multiplier}
\begin{aligned}
&M_m f(\varphi, k):=f(\varphi, k)m(\varphi, k)\quad \mbox{with}\\
& m(\varphi, k):=\sqrt{2\pi}\iunit^k e^{\frac{\iunit\varphi k}{2}} J_{k}\paren{4 |g_{\omega}|\sin\frac{\varphi}{2}},
\end{aligned}
\end{align}
and $\mathcal{G}:\ell^2(\Z^2)\to\ell^2(\Z^2)$ is a matrix shift operator defined by
\begin{equation}\label{operatorG}
(\mathcal{G}\rho)_{n,k}:=\rho_{n+k,n},
\end{equation}
i.e.\ the $k$-th column of $\mathcal{G}\rho$ is the $k$-th diagonal of $\rho$.
\end{prop}

\begin{proof}
By plugging the definition of $U_{\omega}$ into the definition of $T$ and using $U^*_{n,l}=\overline{U_{l,n}}$, we obtain
\begin{equation*}
\begin{split}
(T\rho)(l, \theta) &= (U_\omega(\theta)\rho U^*_\omega(\theta))_{l, l}\\
&= \sum_{m,n}(U_\omega(\theta))_{l, m} \rho_{m,n} (\overline{U^*_\omega(\theta)})_{n,l}\\
&= \sum_{m,n}e^{\iunit(l-m)\theta}J_{l-m}(2|g_\omega|) \rho_{m,n} e^{\iunit(n-l)\theta}\overline{J_{l-n}(2|g_\omega|)}.
\end{split}
\end{equation*}
Applying $\F^{(1,*2)}$ to both sides of this equation yields
\begin{eqnarray*}
\lefteqn{\left({\F^{(1,*2)}} T\rho\right)(\varphi,k)= \frac{1}{2\pi}\perIntegral\sum_{l\in\Z}e^{\iunit\varphi l}e^{\iunit\theta k} (T\rho)(l, \theta)d\theta} \\
&=& \frac{1}{2\pi}\sum_{l,m,n\in\Z}\left(\perIntegral e^{\iunit\theta (n-m+k)}d\theta\right) e^{\iunit\varphi l} J_{l-m}(2|g_\omega|)\overline{J_{l-n}(2|g_\omega|)} \rho_{m,n}.
\end{eqnarray*}
As $\perIntegral e^{\iunit\theta (n-m+k)}d\theta= 2\pi\delta_{m-n.k}$, this simplifies to
\begin{align*}
\left({\F^{(1,*2)}} T\rho\right)(\varphi,k)&= \sum_{\substack{m, n\in \Z \\ m-n=k}}\sum_{l\in \Z} e^{\iunit\varphi l}J_{l-m}(2|g_{\omega}|)\overline{J_{l-n}(2|g_\omega|)} \rho_{m,n}\\
&= \sum_{n, l\in \Z} e^{\iunit\varphi l}J_{l-n-k}(2|g_{\omega}|)\overline{J_{l-n}(2|g_\omega|)} \rho_{n+k,n}\\
&=\paren{\sum_{n'\in\Z}e^{\iunit\varphi n'}J_{n'-k}(2|g_{\omega}|)\overline{J_{n'}(2|g_\omega|)}}
\paren{\sum_{n\in\Z} e^{\iunit\varphi n}\rho_{n+k,n}}
\end{align*}
where we have used the substitution $l=n'+n$ in the last line. Using the identity
\begin{equation*}
\sum_{n\in \Z} e^{\iunit\tau n}\rho_{n+k,n} = \sum_{n\in \Z} e^{\iunit\tau n}(\mathcal{G}\rho)_{n,k}
=\sqrt{2\pi}({\F^{(1)}}(\mathcal{G}\rho))(\varphi,k)
\end{equation*}
and setting
\[
\widetilde{m}(\varphi,k):= \sqrt{2\pi}\sum_{n'\in \Z} e^{\iunit\varphi n'}J_{n'-k}(2|g_{\omega}|)\overline{J_{n'}(2|g_\omega|)}
\]
leads to the formula
\begin{equation}\label{Fouri1}
\left({\F^{(1,*2)}} T\rho\right)(\varphi,k) = \widetilde{m}(\varphi,k) \left({\F^{(1)}}(\mathcal{G}\rho)\right)(\varphi,k).
\end{equation}
In view of \eqref{eq:mixedFTinv} it remains to show that $\tilde{m}=m$.
%
To this end we apply the Fourier convolution theorem \eqref{eq:Fourier_convolution} to obtain
\begin{equation*}
\begin{split}
\sqrt{2\pi}\widetilde{m}(\varphi,k)
&= \perIntegral  \left(\sum_{n'\in \Z} e^{\iunit\tau n'}J_{n'-k}(2|g_{\omega}|)\right) \left(\sum_{n\in \Z} e^{\iunit(\varphi-\tau)n}\overline{J_{n}(2|g_\omega|)}\right) d\tau.\\
&= \perIntegral  e^{\iunit\tau k}\left(\sum_{n'\in \Z} e^{\iunit\tau (n'-k)}J_{n'-k}(2|g_{\omega}|)\right) \left(\overline{\sum_{n\in \Z} e^{\iunit(\tau-\varphi)n}J_{n}(2|g_\omega|)}\right) d\tau.
\end{split}
\end{equation*}
With the help of the identity $\sum_{m}e^{\iunit m\theta}J_m(z)=e^{\iunit z\sin\theta}$ for
Bessel functions (see \cite[eq.~(9.20)]{temme:96}), we obtain
\begin{align*}
\widetilde{m}(\varphi,k)
&= \frac{1}{\sqrt{2\pi}}\perIntegral  e^{\iunit\tau k}e^{2\iunit |g_{\omega}|(\sin\tau-\sin(\tau-\varphi))} d\tau\\
&=  \frac{1}{\sqrt{2\pi}}\perIntegral  e^{\iunit\tau k}e^{4\iunit |g_{\omega}|\sin\frac{\varphi}{2}\cos(\tau-\frac{\varphi}{2})} d\tau \\
&=  \frac{1}{\sqrt{2\pi}}e^{\frac{\iunit\varphi k}{2}}\perIntegral  e^{\iunit\tau k} e^{4\iunit |g_{\omega}|\sin\frac{\varphi}{2}\cos\tau} d\tau.
\end{align*}
Now the identity
 $$\perIntegral e^{\iunit^k\tau+\iunit z\cos\tau} d\tau= 2\pi \iunit^k J_{k}(z)$$
(see \cite[eq.~(9.19)]{temme:96})
shows that $\widetilde{m}=m$ and completes the proof in view of \eqref{Fouri1}.
\end{proof}

The factorization \eqref{eq:Tfactorization} leads us to a proof of the injectivity of $T$.

\begin{proof}[Proof of Theorem \ref{thUniqueness}:]
Since the Fourier-type transforms ${\F^{(1)}}$ and ${\F^{(*1,2)}}$ and the operator $\mathcal{G}$
in the factorization \eqref{eq:Tfactorization}
are all unitary, it suffices to prove the multiplication operator $M$ is injective. To this end, we notice that the multiplier $\sqrt{2\pi}\iunit^k e^{\frac{\iunit\varphi k}{2}} J_{k}(4 |g_{\omega}|\sin\frac{\varphi}{2})$ is a holomorphic function with respect to $\varphi$, so it only has isolated zeros. Therefore, if $Mf=0$ for some $f\in L^2(\perIntv\times\Z)$, we are able to infer that $f$ vanishes almost everywhere. This means $f=0$ in $L^2(\perIntv\times\Z)$, which concludes the proof.
\end{proof}

\begin{cor}[ill-posedness]\label{coro:illposedness}
The inverse of $T$ is unbounded.
\end{cor}

\begin{proof}
It follows from the factorization \eqref{eq:Tfactorization} that
\[
T^{-1}= \mathcal{G}^{-1}\F^{(*1)}M_m^{-1}\F^{(1,*2)}.
\]
All operators on the right hand side are unitary except the multiplication operator $M_m^{-1} = M_{1/m}$.
A multiplication operator on $L^2$ is bounded if and only if the multiplier function is essentially bounded.
Since the Bessel function $J_k$ has zeroes at $0$ for $k\geq 1$, $1/m$ is not essentially bounded.
\end{proof}

\section{Variational source conditions}\label{sec:vsc}
In this section we review variational source conditions and verify conditions of this type,
which then imply both the
stability estimates and the convergence rates in Theorems \ref{rates1} and \ref{rates2}.


\subsection{Basic theory}

We first recall some standard regularization theory for inverse problems.
We consider a general linear ill-posed inverse problem $Tf^\dagger=\yobs$ where
$T:\Xspace\to \Yspace$ is a bounded linear operator between Hilbert spaces,
which does not have a bounded inverse. Let $\yobs$ is the noisy data with
noise level $\|\yobs-Tf^\dagger\|_{\Yspace}<\delta$.
Tikhonov regularization with constraint set $\mathcal{C}\subset\Xspace$
and regularization parameter $\alpha>0$ is given by
\begin{equation}\label{eq:TikhFunc}
\widehat{f}_{\alpha}=\argmin_{f\in\mathcal{C}}\bracket{\|Tf-\yobs\|_\Yspace^2+\alpha\|f\|_\Xspace^2}.
\end{equation}

%
%
%
%
To obtain bounds on the reconstruction error $\|\widehat{f}_{\alpha}-f^\dagger\|$,
which tend to $0$ as the noise level tends to $0$,  we need to impose conditions
on $f^{\dagger}$. Such conditions are usually referred to as
\emph{source condition}. Classically, the source condition is of the form
\[
f^\dagger=h(T^*T)\omega,\qquad\|\omega\|_{\Xspace}\leq E
\]
where $\omega \in \Xspace$ is some ``source'' and
$h:[0,\|T^*T\|]\to[0,\infty)$ is some increasing function,  which determines
the convergence rate. The usefulness of such spectral source conditions
is linked to the applicability of tools from spectral theory which
is mostly restricted to linear reconstruction procedure. Even for
linear reconstruction procedures such as unconstrained Tikhonov regularization
they are only sufficient, but not quite necessary for certain
convergence rates.
Both of these shortcomings can be overcome by the use of
source conditions in the form of variational inequalities
(see \cite{HKPS:07,Scherzer_etal:09,flemming:12,WSH:18}):

\begin{defn}[Variational Source Condition]
A function $\psi:[0,\infty)\to [0,\infty)$ is called an \emph{index function} if
it is continuous, strictly increasing, and $\psi(0)=0$.
We say that $f^\dagger$  satisfies a \emph{variational source condition} with index function $\psi$
if
\begin{equation}\label{eqVSC}
\frac{1}{4}\norm{f-f^\dagger}^2 \leq  \frac{1}{2}\norm{f}^2 - \frac{1}{2}\norm{f^\dagger}^2
+\psi\paren{\norm{T(f)-T(f^\dagger)}^2} \,\,\,\mbox{for all } f \in \Xspace.
\end{equation}
\end{defn}

Conditions of the form \eqref{eqVSC} lead to the following error bounds
in terms of  the index function $\psi$ (see \cite{grasmair:10,flemming:12}):
\begin{prop}[convergence rates with a-priori choice of $\alpha$]\label{prop:convTikh}
Consider Tikhonov regularization given by eq.~\eqref{eq:TikhFunc}.
If $f^{\dagger}\in\mathcal{C}$ satisfies a variational source condition \eqref{eqVSC} with some concave, differentiable index function $\psi$
and if the regularization parameter $\alpha$ is chosen by $\alpha=1/\psi'(4\delta^2)$, then
the following error bound holds true:
\[
\norm{\widehat{f}_{\alpha}-f^\dagger}\leq 4\sqrt{\psi(\delta^2)}.
\]
\end{prop}
Actually eq.~\eqref{eqVSC} only needs to hold for all $f\in\mathcal{C}$.
The same convergence rate can also be achieved by the discrepancy principle,
which does not require prior knowledge of the index function $\psi$
encoding properties of the unknown solution.
\begin{prop}[convergence rates with discrepancy principle]\label{prop:convTikhDisc}
Consider Tikhonov regularization given by eq.~\eqref{eq:TikhFunc}.
If $f^{\dagger}\in\mathcal{C}$ satisfies a variational source condition \eqref{eqVSC} with some concave,
differentiable index function $\psi$
and if the regularization parameter $\alpha$ is chosen according to the discrepancy principle
\eqref{eq:disc}, then the following error bound holds true:
\[
\norm{\widehat{f}_{\alpha}-f^\dagger}\leq 4(1+\tau)\sqrt{\psi(\delta^2)}.
\]
\end{prop}
\begin{proof}
The proof is adapted from {\cite[Theorem 4.3(iii)]{HM:19}}, see also
\cite{flemming:12}.
We first notice that since $\widehat{f}_{\alpha}$ is defined to be the minimizer of the Tikhonov functional \eqref{eq:TikhFunc}, it follows from \eqref{eq:disc} that
\begin{align*}
\norm{\widehat{f}_{\alpha}}^2-\norm{f^\dagger}^2
&\leq\frac{1}{\alpha} \left(\norm{Tf^\dagger-\yobs}_{\Yspace}^2-\norm{T\widehat{f}_{\alpha}-\yobs}_{\Yspace}^2\right)\\
&\leq \frac{1}{\alpha} \left(\delta^2-\delta^2\right)=0.
\end{align*}
Combining this inequality with \eqref{eqVSC} yields
\[
\frac{1}{4}\norm{\widehat{f}_{\alpha}-f^\dagger}^2 \leq \psi\paren{\norm{T\widehat{f}_{\alpha}-Tf^\dagger}^2}.
\]
As
\begin{align*}
\norm{T\widehat{f}_{\alpha}-Tf^\dagger}&\leq\norm{T\widehat{f}_{\alpha}-\yobs}+\norm{\yobs-Tf^\dagger}
\leq \tau\delta+\delta,
\end{align*}
and $\psi$ is monotonically increasing, we obtain $\|\widehat{f}_{\alpha}-f^\dagger\|^2\leq \psi\paren{(1+\tau)^2\delta^2}$.
Now the proof is completed by noting that $\psi\paren{(1+\tau)^2\delta^2}\leq (1+\tau)^2\psi\paren{\delta^2}$ as $\psi$
is concave with $\psi(0)=0$.
\end{proof}
Here we do not address the question whether a parameter $\alpha>0$ satisfying
\eqref{eq:disc} exists. We refer to \cite{AHM:14} for the so-called sequential
discrepancy principle, which determines $\alpha$ by an explicit
algorithm for which similar error bounds can be shown.

\begin{rem}[{\cite[Eq. 6]{HohWei15}}]\label{rem:VSCstab}
If the variational source condition \eqref{eqVSC} is satisfied for all
$f^{\dagger}$ in some subset $\mathcal{K}\subset\Xspace$, then
the conditional stability estimate
\[
\|f_1-f_2\|\leq 2\sqrt{\psi\left(\|Tf_1-Tf_2\|^2\right)},
\]
holds true for all $f_1,f_2\in \mathcal{K}$.
\end{rem}

To verify variational source conditions for our problem, we will check the
sufficient conditions in the following lemma, which is a special case
of \cite[Theorem 2.1]{HW:17} (where a different scaling is used such that
the index function $\psi$ in \cite{HW:17} is twice the Bessel function $\psi$ here):
\begin{lemma}[Verification of VSCs]\label{le:VSC}
Let $\Xspace$ and $\Yspace$ be Hilbert spaces and $V_{\varepsilon} \in \Xspace$ be a family of subspaces. Suppose that there exists a family of orthogonal projection operators $P_{\varepsilon}: \Xspace \rightarrow V_{\varepsilon}$ for $\varepsilon$ in some index set $\mathcal{I}\subset (0,\infty)$,
and there exist families $(\kappa_{\varepsilon})_{\varepsilon\in \mathcal{I}}$, $(\sigma_{\varepsilon})_{\varepsilon\in\mathcal{I}}$ of positive numbers, such that the following conditions hold true:
\begin{itemize}
\item $\|(I-P_{\varepsilon}) f^\dagger\|_{\Xspace} \leq \kappa_{\varepsilon}$ for all
$\varepsilon\in\mathcal{I}$;
\item  $\inf_{\varepsilon\in\mathcal{I}} \kappa_{\varepsilon}=0$;
\item For all $f\in \Xspace$ and all $\varepsilon\in\mathcal{I}$
we have
\[\langle P_{\varepsilon} f^\dagger, f^\dagger-f \rangle \leq \sigma_{\varepsilon} \|T(f^\dagger)-T(f)\|_\Yspace.
 \]
\end{itemize}
Then the true solution $f^\dagger$ satisfies the variational source condition \eqref{eqVSC} with the index function
\begin{align}\label{eq:psiVSC}
\psi(\tau) := \inf_{\varepsilon\in\mathcal{I}}[\sigma_{\varepsilon}\sqrt{\tau}+ \kappa_{\varepsilon}^2].
\end{align}
\end{lemma}

\subsection{Tools for the verification of variational source conditions for $T$}

In order to verify the VSC for our forward operator $T$, we first look at the decomposition in Proposition~\ref{prop:Tfactorization}. The fact that ${\F^{(1)}},{\F^{(1,*2)}}$ and $\mathcal{G}$ are all unitary operators implies that, in order to analyze the properties of the operator $T$, it suffices to analyze the properties of the multiplication operator $M_m$ as defined in \eqref{eq:Multiplier}. We write the forward problem $T\rho^\dagger=\yobs$ as
\[
M_mf^\dagger=\F^{(1,*2)}\yobs,
\]
where $f^\dagger$ is defined as
\begin{equation}\label{eq:function-f}
f^\dagger:={\F^{(1)}}\mathcal{G}\rho^\dagger\in L^2(\perIntv\times\Z),
\end{equation}
and will verify the three conditions in Lemma~\ref{le:VSC} for the multiplication operator $M_m$ and the true solution $f^\dagger$.

For any $\varepsilon>0$, we define the sublevel sets
\begin{align}\label{def:IntervalI}
I_{\varepsilon}&:=\{(\varphi,k)\in\Omega:|m(\varphi,k)|< \varepsilon\},\\
I_{k,\varepsilon}&:=\left\{\varphi\in\perIntv\middle|\,|m(\varphi, k)|< \varepsilon\right\}
\end{align}
with $\Omega:=\perIntv\times \Z$ such that $|I_{\varepsilon}| = \sum_{k\in\Z} |I_{k,\varepsilon}|$.
We choose $P_\varepsilon$ as orthogonal projections  from $L^2(\Omega)$ to
$L^2(\Omega\setminus I_\varepsilon)$.
Obviously, $P_{\varepsilon}$ can be written as a multiplication operator
\begin{equation}\label{eq:defi_Pepsi}
P_{\varepsilon} f = (1-\chi_{I_\varepsilon})f
\end{equation}
with the characteristic function $\chi_{I_\varepsilon}$ of $I_{\varepsilon}$.

To bound the sizes of the sublevel sets $I_{k,\varepsilon}$ of $m(\cdot,k)$,
we recall some properties of the Bessel functions
$J_k$ which can be found in \cite[Ch.~9]{temme:96}: 
\begin{enumerate}
  \item $J_{-k}(z)=(-1)^kJ_k(z)$. Hence, it suffices to look at the case $k\geq0$.
  \item For any $k>0$, $J_{k}$ has a zero of order $k$ at $z=0$. Around this zero it has the asymptotic behavior $J_{k}(z)= \frac{z^{k}}{2^{k}k!}(1+\calO(|z|))$ as $|z|\to 0$.
  \item $J_k$ also possesses an infinite number of simple zeroes $0<j_{k,1}<j_{k,2}<\dots$ on the positive real axis, and $\lim_{l\to\infty}j_{k,l}=\infty$ for all $k\in\Z$.
  \item The positions of the positive zeroes tend to infinity as $k\to+\infty$, i.e.\
	$\lim_{k\to+\infty} j_{k,1}=+\infty$.
\end{enumerate}
These properties translate into the following facts on $m(\cdot,k)$ and its
zero set $I_{k,0}$:
\begin{lemma}\label{le:mProperties}
Let $m(\varphi,k)$ be defined as in \eqref{eq:Multiplier}. Then we have the following:
\begin{enumerate}
  \item $|m(\varphi,-k)|=|m(\varphi,k)|$ and $I_{k,0}=I_{-k,0}$ for all $k\in\Z$ and
	$\varphi\in \perIntv$.
  \item $0\in I_{k,0}$ for all $k>0$, and
	$|m(\varphi,k)| = \frac{|g_\omega|^k\varphi^{k}}{k!}(1+\calO(|\varphi|))$ as $|\varphi|\to 0$.
  \item $I_{k,0}\setminus \{0\}$ is a finite set of simple zeros of $m(\cdot,k)$
	for all $k\in \Z$.
	\item There exists $K(g_{\omega})\in\N$ such that $I_{k,0}=\{0\}$ for all
	$|k|\geq K(g_{\omega})$.
\end{enumerate}
\end{lemma}
\begin{proof}
In view of the expression \eqref{eq:Multiplier} for $m$,
the first three statements are immediate consequences of the first three properties of $J_k(z)$.
For the last statement, note that $I_{k,0}=\{0\}\Leftrightarrow j_{k,1}\leq 4 |g_{\omega}|$.
As $\lim_{k\to+\infty} j_{k,1}=+\infty$, there exists $K(g_{\omega})$ such that
$j_{k,1}> 4 |g_{\omega}|$ for all $|k|\geq K(g_{\omega})$.
\end{proof}


We further have the following properties of $I_{k,\varepsilon}$ as $\varepsilon\to0$:
\begin{lemma}\label{le:IfkLength}
\begin{enumerate}
\item For all $k\in\Z$ there exists $\varepsilon_0(k)>0$ such that for all $0<\varepsilon\leq \varepsilon_0(k)$  each connected component of $I_{k,\varepsilon}$ contains exactly one point of $I_{k,0}$.
  \item For $k\neq0$, the connected component of $I_{k,\varepsilon}$ containing $0$ has
	size  $\frac{k}{e|g_\omega|}\varepsilon^{1/|k|}(1+\calO(1))$ as $\varepsilon\to 0$.
  \item For all $|k|<K(g_{\omega})$, the finitely many connected components of $I_{k,\varepsilon}$
	not containing $0$ are of size $\calO(\varepsilon)$ as $\varepsilon\to 0$.
  \item There exists a constant $C_I>0$ depending only on $|g_{\omega}|$ such that
\begin{subequations}\label{eqs:sizeI}
	\begin{align}
	\label{eq:sizeIk}
	&|I_{k,\varepsilon}|\leq \min\paren{C_Ik\varepsilon^{1/|k|},2\pi} 
	+ \calO(\varepsilon),\qquad \mbox{for }k\in\Z.\\
	\end{align}
	\end{subequations}
\end{enumerate}
\end{lemma}
\begin{proof}
1.) Since $m(\cdot,k)$ has only finitely many zeros of finite order, there exists
$\varepsilon_0(k)$ such that $m(\cdot,k)$ is monotonic (or even and monotonic on $[0,\pi]$ in case of
$m(\cdot,2k)$ around $\varphi=0$)
on all connected components of $\Omega_{\varepsilon_0(k)}$
viewed as subset of $\R/(2\pi\Z)$. Since $\perIntv\setminus \Omega_{\varepsilon_0(k)}$ is bounded,
$|m(\cdot,k)|$ attains its infimum on the closure of this set, and the infimum is positive.
Thus we obtain the claim by possibly reducing $\varepsilon_0(k)$. \\
2,3) This follows from Lemma \ref{le:mProperties}, parts 2 and 3, respectively.
In part 2 we also use the Stirling approximation
$(k!)^{1/k}= \frac{k}{e}\paren{1+o(1)}$ as $k\to\infty$.\\
4) This follows from parts 1--3.
\end{proof}

Moreover, we need a bound on the supremum norm of $f^{\dagger}$:
\begin{lemma}\label{Lemma:F1GRhoBounded}
Let $\rho$ be an arbitrary density matrix, i.e.\ $\rho$ is self-adjoint and positive semi-definite, with trace equal to 1. Let the operators ${\F^{(1)}}$ and $\mathcal{G}$ be defined as in \eqref{FourierTransformF1} and \eqref{operatorG}, respectively. Then
\begin{equation*}
\left\|f^{\dagger}\right\|_{L^{\infty}} =
\left\|{\F^{(1)}}\mathcal{G}\rho\right\|_{L^\infty} \leq\frac{1}{\sqrt{2\pi}}.
\end{equation*}
\end{lemma}
\begin{proof}
Using the definitions of ${\F^{(1)}}$ and $\mathcal{G}$, we can see that
\begin{equation*}
\left|{\F^{(1)}}\mathcal{G}\rho(\varphi,k)\right|=\frac{1}{\sqrt{2\pi}}\left|\sum_{n\in\Z}e^{i\varphi n}\rho_{n+k,n}\right|
\leq\frac{1}{\sqrt{2\pi}}\sum_{n\in\Z}\left|\rho_{n+k,n}\right|
\end{equation*}
for all $\varphi\in\perIntv$ and all $k\in\Z$.
The fact that $\rho$ is positive semi-definite implies that for all $n,k\in\Z$, the principal submatrix
\[
\left(
  \begin{array}{cc}
    \rho_{n,n} & \rho_{n,n+k} \\
    \rho_{n+k,n} & \rho_{n+k,n+k} \\
  \end{array}
\right)
\]
is also positive semi-definite. Therefore, calculating the determinant of this submatrix yields that
\[
0\leq\det\left(
  \begin{array}{cc}
    \rho_{n,n} & \rho_{n,n+k} \\
    \rho_{n+k,n} & \rho_{n+k,n+k} \\
  \end{array}
\right)
=\rho_{n,n}\rho_{n+k,n+k}-|\rho_{n+k,n}|^2,
\]
so we conclude using Young's inequality that
\[
|\rho_{n+k,n}|\leq\sqrt{\rho_{n,n}\rho_{n+k,n+k}}\leq\frac12\left(\rho_{n,n}+\rho_{n+k,n+k}\right)
\]
holds for all $n,k\in\Z$. This in turn implies that
\begin{equation*}
\sum_{n\in\Z}\left|\rho_{n+k,n}\right|\leq\sum_{n\in\Z}\frac12\left(\rho_{n,n}+\rho_{n+k,n+k}\right)
=\frac12\left(\tr \rho+\tr \rho\right)=1,
\end{equation*}
which concludes the proof.
\end{proof}

Note that the third condition in Lemma~\ref{le:VSC} reduces to 
$\langle P_\varepsilon f^\dagger, \tilde{f} \rangle \leq \sigma_\varepsilon \|M_m \tilde{f}\|$ for all $\tilde{f}\in L^2(\Omega)$. Using the Cauchy-Schwarz inequality and elementary
estimates this condition can easily be verified with $\sigma_\varepsilon=\varepsilon^{-1}\|f^\dagger\|_{L^2}$.
However, the following more elaborate argument along the lines of \cite[Theorem 3.1]{HW:17}
gives a sharper bound, which is optimal in some sense (see *****):
\begin{lemma}\label{lem:sigmaeps}
Suppose the first condition in Lemma~\ref{le:VSC} holds true for the projection operators
defined in \eqref{eq:defi_Pepsi} and that $\varepsilon\mapsto
\kappa_{\varepsilon}\varepsilon^{\nu-1}$ is decreasing for some $\nu\in (0,1)$.
Then the third condition holds true with
\[
\sigma_{\varepsilon} = \frac{\kappa_{\varepsilon}}{\sqrt{6\nu}\varepsilon}\,.
\]
\end{lemma}

\begin{proof}
Due to the inequality
\begin{align*}
\langle P_\varepsilon f^\dagger, \tilde{f}\rangle
&= \int_{\Omega\setminus I_\varepsilon} \frac{f^{\dagger}}{m}(m\cdot \tilde{f}) dx
\leq \left\|\frac{f^{\dagger}}{m}\right\|_{L^2(\Omega\setminus I_\varepsilon)}
\|M_m \tilde{f}\|_{L^2(\Omega)},
\end{align*}
we have to show that
\begin{align}\label{eq:toshow}
\left\|\frac{f^{\dagger}}{m}\right\|_{L^2(\Omega\setminus I_\varepsilon)} \leq \sigma_{\varepsilon}.
\end{align}
Introducing the function
\[
\mu_{f^{\dagger}}(\varepsilon) :=  \int_{I_\varepsilon}|f^{\dagger}|^2dx
= \|(I-P_{\varepsilon})f^{\dagger}\|_{L^2}^2,\qquad \varepsilon>0,
\]
which by assumption is bounded by $\mu_{f^{\dagger}}(\varepsilon)\leq \kappa_{\varepsilon}^2$,
we obtain
\begin{align*}
\left\|\frac{f^{\dagger}}{m}\right\|_{L^2(\Omega\setminus I_\varepsilon)}^2
&= \int_{\Omega\setminus I_\varepsilon} \frac{|f^{\dagger}|^2}{|m|^2}\,dx
= \int_\varepsilon^{\infty} \frac{1}{t^2} d\mu_{f^{\dagger}}(t)\\
&= - \frac{\mu_{f^{\dagger}}(\varepsilon)}{\varepsilon^2}
+ \int_\varepsilon^{\infty} \frac{\mu_{f^{\dagger}}(t)}{3t^3}\,dt
\leq \int_\varepsilon^{\infty} \frac{\kappa_t^2}{3t^3}\,dt
\end{align*}
using a partial integration in the last line and the fact that
$\mu_{f^{\dagger}}(t)=\|f^{\dagger}\|_{L^2}^2$ for $t>\|m\|_{L^{\infty}}$ such that
$\lim_{t\to\infty}t^{-2}\mu_{f^{\dagger}}(t)=0$.
Now we use the assumption that $t\mapsto \kappa_t^2 t^{2\nu-2}$ is decreasing to estimate
\begin{align*}
\int_\varepsilon^{\infty} \frac{\kappa_t^2}{3t^3}\,dt
&=  \int_\varepsilon^{\infty} \frac{\kappa_t^2}{3t^{2-2\nu}}\frac{dt}{t^{1+2\nu}}
\leq \frac{\kappa_{\varepsilon}^2}{3\varepsilon^{2-2\nu}}
\int_{\varepsilon}^{\infty}  \frac{dt}{t^{1+2\nu}}
= \frac{\kappa_{\varepsilon}^2}{6\nu\varepsilon^2}.
\end{align*}
This completes the proof of \eqref{eq:toshow}.
\end{proof}

\subsection{Proofs of Theorems \ref{rates1} and \ref{rates2}}

\begin{prop}[H\"older VSC for Theorem \ref{rates1}]\label{prop:VSC1}
Under the assumptions of Theorem \ref{rates1}
the matrix $\rho^\dagger$ satisfies a variational source condition \eqref{eqVSC} with the index function
$\psi(\tau)=C\min\paren{\tau^{\frac{1}{1+2k_0}},\sqrt{\tau}}$ and some $C>0$ depending only on $k_0$ and $|g_{\omega}|$.
\end{prop}

\begin{proof}

From the band-limited assumption of $\rho^\dagger$, we know $f^\dagger(\varphi, k)=0$ for $|k|>k_0$. Therefore, it follows from Lemmas \ref{Lemma:F1GRhoBounded}
and \ref{le:IfkLength} that
\begin{align*}
\|(I-P_{\varepsilon})f^\dagger\|_{L^2}^2
&\leq \|f^{\dagger}\|_{L^{\infty}}^2\sum_{|k|\leq k_0}\left|I_{k,\varepsilon}\right| \\
&\leq \frac{1}{2\pi}\paren{C_0\varepsilon + 2\sum_{k=1}^{k_0} \min\paren{C_Ik\varepsilon^{1/|k|},0}}\\
&\leq \tilde{C}\varepsilon^{1/k_0}
\end{align*}
for all $\varepsilon\leq \varepsilon_0:=\min\{\varepsilon_0(0),\dots,\varepsilon_0(k_0)\}$
and some constant $\tilde{C}>0$ depending only on $k_0$ and $|g_{\omega}|$.

Thus the first and second conditions of Lemma \ref{le:VSC} are satisfied for
$\kappa_\varepsilon = \tilde{C}^{1/2}\varepsilon^{\frac{1}{2k_0}}$. Note that
the function $\varepsilon\mapsto \kappa_\varepsilon\varepsilon^{-2/3}$ is decreasing as $k_0\geq 1$.
Using Lemma~\ref{lem:sigmaeps} with $\nu=1/3$ we deduce that the third
condition holds true for $\sigma_{\varepsilon} = \frac1{\sqrt2}\varepsilon^{-1}\kappa_{\varepsilon}$.
This implies a variational source condition with index function
\[
\psi(\tau) = \inf_{0<\varepsilon\leq \varepsilon_0}
\left[\sqrt{\frac{\tilde{C}}{2}}\sqrt{\tau}
\varepsilon^{\frac{1}{2k_0}-1}+\tilde{C}\varepsilon^{\frac{1}{k_0}}\right].
\]
If we choose
$\varepsilon=\min\Big(\tau^{k_0/(1+2k_0)},\varepsilon_0\Big)$, we obtain
$\psi(\tau)=C_1\tau^{1/(1+2k_0)}$ for
$\tau<\varepsilon_0^{(1+2k_0)/k_0}$
and $\psi(\tau)=C_2\sqrt{\tau}+C_3$ else
with positive constants $C_1$, $C_2$, $C_3$ depending
only on $|g_\omega|$ and $k_0$.
This yields the assertion.
%
%
%
%
\end{proof}

\begin{proof}[Proof of Theorem \ref{rates1}]
The statement follows from the H\"older-type variational source condition in Proposition \ref{prop:VSC1}.
In particular for the stability estimate \eqref{eq:Hoelder_stability}
we note that $\trace(\rho)=1$ and $\rho\geq 0$
imply that all eigenvalues $\lambda_j$ of $\rho$ lie
in the interval $[0,1]$, and hence
$\|\rho\|_{l^2(\mathbb{Z}^2)}^2 = \sum_{j=0}^\infty \lambda_j^2
\leq \sum_{j=0}^\infty\lambda_j =\trace(\rho)=1$.
Hence, the stability estimate \eqref{eq:Hoelder_stability} follows from
Remark \ref{rem:VSCstab} as
only the behavior of $\psi(\tau)$ of small $\tau$ is
relevant.
The H\"older-type convergence rate \eqref{eq:Hoelder_rate}
follows from  Proposition \ref{prop:convTikh}.
\end{proof}

\begin{prop}[VSC for polynomial decay]\label{prop:VSC_pol_decay}
Under the assumptions of Theorem \ref{rates2}, case
\eqref{Log-Assum}
the matrix $\rho^\dagger$ satisfies a logarithmic variational source condition \eqref{eqVSC} with
\begin{align}\label{eq:psiPolDecay}
\psi(\tau)=C\paren{\frac{-\log\tau}{\log(-\log\tau)}}^{-4\mu}
\end{align}
for some $C>0$ depending only on $C_\rho$, $\mu$ and $|g_{\omega}|$.
\end{prop}

\begin{proof}
The decay condition on $\rho^\dagger$ implies corresponding bounds on $f^\dagger$:
\begin{align*}
|f^\dagger(\varphi, k)|&=\frac{1}{\sqrt{2\pi}}\left|\sum_{n=-\infty}^{\infty}e^{i n\varphi}\rho_{n+k,n} \right|
\leq \frac{1}{\sqrt{2\pi}} \sum_{n=-\infty}^{\infty}|\rho_{n+k,n}|
\leq \frac{C_\rho}{\sqrt{2\pi}}|k|^{-\frac{1}{2}-2\mu}.
\end{align*}
Therefore, considering that $|I_{0,\varepsilon}|=\calO(\varepsilon)$, we have
\begin{equation*}
\|(I-P_{\varepsilon})f^\dagger\|_{L^2}^2 \leq \frac{\tilde{C}_\rho^2}{2\pi}\sum_{k=1}^{\infty} \left|I_{k,\varepsilon}\right| k^{-1-4\mu} + \calO(\varepsilon)
\end{equation*}
for some $\tilde{C_\rho}>0$.

We will utilize both upper bounds of $\left|I_{k,\varepsilon}\right|$ from Lemma~\ref{le:IfkLength}. As $\lim_{k\to\infty} k\varepsilon^{1/k}=\infty$,  we get the bound
$\left|I_{k,\varepsilon}\right|\leq2\pi$ for large $k$.
For some cut-off index $k_0(\varepsilon)$ to be determined
later, we bound the sum by
\[
\|(I-P_{\varepsilon})f^\dagger\|_{L^2}^2 \leq \frac{\tilde{C}_\rho^2C_I}{2\pi}\sum_{k=1}^{k_0(\varepsilon)}k^{-4\mu}\varepsilon^{1/k}+\tilde{C}_\rho^2
\sum_{k=k_0(\varepsilon)+1}^{\infty}k^{-1-4\mu}
+\calO(\varepsilon).
\]
Using the logarithmic derivative
$\frac{d}{dk}\ln(k^{-\mu}\varepsilon^{1/k}) = -4k^{-1}\mu -k^{-2}\ln\varepsilon$, it can be seen that
$k\mapsto k^{-\mu}\varepsilon^{1/k}$ is increasing on the interval
$\bracket{0,k_1(\varepsilon)}$ with
$k_1(\varepsilon):=\frac{1}{4\mu}\ln\frac{1}{\varepsilon}$.  Therefore, as long as $k_0(\varepsilon)\leq k_1(\varepsilon)$,
there exist constants $\varepsilon_0$ and $C>0$ depending
only on $C_\rho$, $|g_{\omega}|$, and $\mu$ such that
\begin{align*}
\|(I-P_{\varepsilon})f^\dagger\|_{L^2}^2
&\leq C\paren{k_0^{-4\mu+1}\varepsilon^{1/k_0}+k_0^{-4\mu}}
\end{align*}
for $\varepsilon\leq \varepsilon_0$. 
We choose $k_0$ such that both terms on the right hand side are
approximately equal, i.e.\
$k_0\varepsilon^{1/k_0}\approx 1$
or equivalently $k_0\ln k_0\approx \ln \frac{1}{\varepsilon}$.
Solving $k_0\ln k_0=y$ for $k_0$ yields the asymptotic relation
\[
k_0 = \frac{y}{\ln k_0}
=  \frac{y}{\ln y-\ln(\ln k_0)}
= \frac{y}{\ln y}\paren{1+o(1)}
\qquad \mbox{as } y\to\infty.
\]
Therefore, we set $k_0(\varepsilon):=
\lfloor\frac{-\log \varepsilon}{\log(-\log\varepsilon)}\rfloor$.
Note that $k_0(\varepsilon)\leq k_1(\varepsilon)$ for $\varepsilon$ sufficiently small. This yields $\|(I-P_{\varepsilon})f^\dagger\|_{L^2}\leq \kappa_\varepsilon$ with
\[
\kappa_\varepsilon
= \calO\paren{\paren{\frac{-\log \varepsilon}{\log(-\log \varepsilon)}}^{-2\mu}},
\] and
the first and second conditions of Lemma \ref{le:VSC} are satisfied. By Lemma \ref{lem:sigmaeps} the third condition
holds true with
$\sigma_{\varepsilon}= \kappa_{\varepsilon}/\varepsilon$.
Therefore, Lemma \ref{le:VSC} yields a VSC with 
$\psi(\tau) = \inf_{0<\varepsilon\leq \varepsilon_0}
[\kappa_{\varepsilon}\sqrt{\tau}/\varepsilon+\varepsilon^2]$.
Choosing $\varepsilon=\min(\tau^{1/3},\varepsilon_0)$, 
the first term is asymptotically neglectible against the second, and we obtain  
a VSC with 
\(
\psi(\tau)=\calO\paren{\kappa_{\tau}^2}\). 
This completes the proof. 
%
\end{proof}

\begin{prop}[VSC for exponential decay]
\label{prop:VSC_exp_decay}
Under the assumptions of Theorem \ref{rates2}, case
\eqref{Hol-Assum2}
the matrix $\rho^\dagger$ satisfies a logarithmic variational source condition \eqref{eqVSC} with $\psi(\tau)=Ce^{-2\sqrt{(-\log b)(-\log\tau)}}$ for some $C>0$ depending only on $C_\rho$, $b$ and $|g_{\omega}|$.
\end{prop}

\begin{proof}
The decay condition on $\rho^\dagger$ implies corresponding bounds on $f^\dagger$:
\begin{align*}
|f^\dagger(\varphi, k)|&=\frac{1}{\sqrt{2\pi}}\left|\sum_{n=-\infty}^{\infty}e^{i n\varphi}\rho_{n+k,n} \right|
\leq \frac{1}{\sqrt{2\pi}} \sum_{n=-\infty}^{\infty}|\rho_{n+k,n}|
\leq \frac{C_\rho}{\sqrt{2\pi}}b^{|k|}.
\end{align*}
Therefore, again noting that $|I_{0,\varepsilon}|=\calO(\varepsilon)$, we have 
\begin{equation*}
\|(I-P_{\varepsilon})f^\dagger\|_{L^2}^2 \leq \frac{\tilde{C_\rho}}{{2\pi}}\sum_{k=1}^{\infty} \left|I_{k,\varepsilon}\right| b^{2|k|} 
+ \calO\paren{\varepsilon}
\end{equation*}
for some $\tilde{C_\rho}>0$.
We will utilize both upper bounds of $\left|I_{k,\varepsilon}\right|$ from Lemma~\ref{le:IfkLength}. As $\lim_{k\to\infty} k\varepsilon^{1/k}=\infty$,  we use the trivial bound $\left|I_{k,\varepsilon}\right|\leq2\pi$ for large $k$.
We choose a cut-off $k_0(\varepsilon)=\lfloor \sqrt{-\log \varepsilon/(-2\log b)}\rfloor$ for $0<\epsilon<1$ and 
obtain
\begin{align*}
\|(I-P_{\varepsilon})f^\dagger\|_{L^2}^2 
&\leq \frac{\tilde{C_\rho}C_I}{\pi}\sum_{k=1}^{k_0(\varepsilon)} k\varepsilon^{1/k}b^{2|k|}+2\tilde{C_\rho}\sum_{k=k_0(\varepsilon)+1}^{\infty}b^{2|k|} 
+ \calO(\varepsilon)\\
&\leq \frac{\tilde{C_\rho}C_I}{\pi}\varepsilon^{\frac{1}{k_0(\varepsilon)}}\sum_{k=1}^{\infty}kb^{2|k|}+ 2\tilde{C_\rho}b^{2k_0(\varepsilon)}\sum_{k=0}^{\infty}b^{-2k}+ \calO\paren{\varepsilon}\\
&\leq C \left(\varepsilon^{\sqrt{\frac{-2\log b}{-\log \varepsilon}}}
+ b^{2\sqrt{\frac{-\log \varepsilon}{-2\log b}}}
\right)
+ \calO\paren{\varepsilon}
\end{align*}
for some generic constant $C>0$ that depends only on $b$, $C_\rho$ and $|g_\omega|$.
Taking the logarithm of the first two terms shows 
that for our choice of $k_0(\varepsilon)$ both logarithms 
equal $-\sqrt{2(-\log b)(-\log\varepsilon)}$. 
This shows that  
\begin{align*}
\|(I-P_{\varepsilon})f^\dagger\|_{L^2}^2
&\leq C e^{-\sqrt{2(-\log b)(-\log\varepsilon)}}
+ \calO\paren{\varepsilon}.
\end{align*}
Therefore, the first and second conditions of Lemma \ref{le:VSC} are satisfied for $\kappa_\varepsilon = C \exp\paren{-\sqrt{\frac{1}{2}(-\log b)(-\log\varepsilon)}}$. Note that
the function $\epsilon\mapsto \varepsilon^{-1/3}\kappa^{\epsilon}$ is decreasing, Lemma~\ref{lem:sigmaeps} with $\nu=\frac23$ allows us to deduce that the third condition holds true for $\sigma_{\varepsilon} = \frac{\kappa_\varepsilon}{2\varepsilon}$.
This implies a variational source condition with index function
\[
\psi(\tau) = \inf_{\varepsilon\in\mathcal{I}}
\left[\frac{\kappa_\varepsilon}{2\varepsilon}\sqrt{\tau}+ \kappa_{\varepsilon}^2\right].
\]
If we choose $\varepsilon=\tau^{\frac14}$ for $0<\tau<1$, we obtain
\[
\psi(\tau)\leq2\kappa_{\tau^{1/4}}\tau^{\frac14}+\kappa_{\tau^{1/4}}^2=\calO\paren{e^{-\sqrt{2(-\log b)(-\log\tau)}}}
\]
as $\tau\searrow 0$.
This yields the assertion.
\end{proof}

\begin{proof}[Proof of Theorem \ref{rates2}]
We set $\Phi(t) = 2\sqrt{\psi(t^2)}$ with the functions
$\psi$ in the variational source conditions of Propositions
\ref{prop:VSC_pol_decay} and \ref{prop:VSC_exp_decay}.
Then the statement follows from
Proposition \ref{prop:convTikh} and Remark \ref{rem:VSCstab}.
\end{proof}

\section{Conclusions}\label{sec:conclusions}
We have shown that the data acquired in the SQUIRRELS
method (without noise and modelling errors) are indeed sufficient to uniquely determine the unknown electron
density matrix. Moreover, we have estimated the
intrinsic difficulty (or degree of ill-posedness) of the inverse
problem to reconstruct a density matrix from these data
under noise.
As expected, the answer strongly depends on the type
of available a-priori information on the unknown
density matrix. If this matrix is band-limited, we
obtain H\"older rates, whereas under polynomial decay
conditions only logarithmic rates can be shown.
For the most realistic exponential decay conditions
the rates are in between H\"older and logarithmic rates.

We conjecture that both the stability estimates and the  convergence rates
are of optimal order under the given a-priori
information if $T$ is considered as an operator defined
on all bounded, Hermitian matrices. However, it is possible that
the positive semidefiniteness constraint, which has a
strong regularizing effect in numerical experiments,
may be further explointed to improve rates.

Another topic of further research in this direction
may be to extend the analysis of this paper to a
model involving a continuum of energy states.

\bibliographystyle{siam}
\bibliography{electron}

\end{document}